\documentclass[11pt,reqno]{amsart}
\usepackage{amsmath, amsthm, amscd, amsfonts, amssymb, graphicx, color}
\usepackage[colorlinks]{hyperref}
\newtheorem{thm}{Theorem}[section]
\newtheorem{cor}[thm]{Corollary}

\newtheorem{prop}[thm]{Proposition}
\theoremstyle{definition}

\theoremstyle{remark}

\numberwithin{equation}{section}

\begin{document}
\title[Approximately quadratic mappings on restricted domains]
      {Approximately quadratic mappings on restricted domains}%
\author[A. Najati, S.-M. Jung]{Abbas Najati and Soon-Mo Jung}
\address{\noindent Abbas Najati \newline
\indent Department of Mathematics
\newline
\indent Faculty of Sciences
\newline
\indent   University of Mohaghegh Ardabili
\newline \indent  Ardabil
\newline \indent Iran}
\email{a.nejati@yahoo.com}
\address{\noindent Soon-Mo Jung \newline
\indent Mathematics Section
\newline
\indent College of Science and Technology
\newline
\indent   Hongik University
\newline \indent 339-701 Jochiwon
\newline \indent Republic of Korea}
\email{smjung@hongik.ac.kr}


\begin{abstract}
In this paper, we introduce a generalized quadratic functional
equation
\begin{equation*}
f(rx+sy) = rf(x)+sf(y)-rsf(x-y)
\end{equation*}
where $r,\, s$ are nonzero real numbers with $r+s=1$.
We show that this functional equation is quadratic if $r,\, s$
are rational numbers.
We also investigate its stability problem on restricted domains.
These results are applied to study of an asymptotic behavior of
these generalized quadratic mappings.
\end{abstract}
\maketitle

\section{Introduction}

\textit{Under what conditions does there exist a group
homomorphism near an approximate group homomorphism?}
This question concerning the stability of group homomorphisms
was posed by Ulam \cite {ul60}.
The case of approximately additive mappings was solved by
Hyers \cite{hy41} on Banach spaces.
In 1950 Aoki \cite{aok} provided a generalization of the Hyers'
theorem for additive mappings and in 1978 Rassias \cite{ras78}
generalized the Hyers' theorem for linear mappings by allowing
the Cauchy difference to be unbounded (see also \cite{Bou51}).
The result of Rassias' theorem has been generalized by
G\u{a}vruta \cite{ga94} who permitted the Cauchy difference to
be bounded by a general control function.
This stability concept is also applied to the case of other
functional equations.
For more results on the stability of functional equations,
see \cite{Chole}, \cite{Czerwik92}, \cite{FRS}, \cite{For1},
\cite{For2}, \cite{Grabiec}, \cite{HYR}--\cite{Jun},
\cite{Kann95}--\cite{pa02} and \cite{Ras.91}--\cite{Ras.000}.
We also refer the readers to the books \cite{Acz89}, \cite{cz02},
\cite{hir98}, \cite{jung} and \cite{R.b2}.
\par
It is easy to see that the quadratic function $f(x)=x^2$ is a
solution of each of the following functional equations:
\begin{align}
f(x+y)+f(x-y)&=2f(x)+2f(y), \label{q}
\\
f(rx+sy)+rsf(x-y)&=rf(x)+sf(y) \label{Gq}
\end{align}
where $r,\, s$ are nonzero real numbers with $r+s=1$.
So, it is natural that each equation is called a quadratic
functional equation.
In particular, every solution of the quadratic equation
(\ref{q}) is said to be a quadratic function.
It is well known that a function $f : X \to Y$ between real
vector spaces $X$ and $Y$ is quadratic if and only if there
exists a unique symmetric bi-additive function
$B : X \times X \to Y$ such that $f(x)=B(x,x)$ for all
$x \in X$ (see \cite{Acz89}, \cite{hir98}, \cite{Kann95}).
\par
We prove that the functional equations (\ref{q}) and (\ref{Gq})
are equivalent if $r,\, s$ are nonzero rational numbers.
The functional equations (\ref{q}) is a spacial case of
(\ref{Gq}).
Indeed, for the case $r=s=\frac{1}{2}$ in (\ref{Gq}), we get
(\ref{q}).
\par
In 1983 Skof \cite{Skof3} was the first author to solve the
Hyers-Ulam problem for additive mappings on a restricted domain (see
also \cite{HIR98},\cite{S.M.Jung98} and \cite{J.M.S}). In 1998 Jung
\cite{Jung98} investigated the Hyers-Ulam stability for additive and
quadratic mappings on restricted domains (see also \cite{jung0},
\cite{jungk} and \cite{jungs}). Rassias \cite{JRassias02}
investigated the Hyers-Ulam stability of mixed type mappings on
restricted domains.

\section{Solutions of EQ. (\ref{Gq})}

In this section we show that the functional equation (\ref{Gq})
is equivalent to the quadratic equation (\ref{q}).
That is, every solution of Eq. (\ref{Gq}) is a quadratic
function.
We recall that $r,\, s$ are nonzero real numbers with $r+s=1$.

\begin{thm}\label{thm1}
Let $X$ and $Y$ be real vector spaces and $f : X \to Y$ be an
odd function satisfying (\ref{Gq}).
If $r$ is a rational number, then $f \equiv 0$.
\end{thm}
\begin{proof}
Since $f$ is odd, $f(0)=0$.
Letting $x=0$ (respectively, $y=0$) in (\ref{Gq}), we get
\begin{equation}\label{pthm1.1}
f(sy) = s(1+r)f(y), \quad  f(rx) = r^2f(x)
\end{equation}
for all $x, y \in X$.
Replacing $y$ by $-y$ in (\ref{Gq}) and adding the obtained
functional equation to (\ref{Gq}), we get
\begin{equation}\label{pthm1.2}
f(rx+sy)+f(rx-sy)=2rf(x)-rs[f(x+y)+f(x-y)]
\end{equation}
for all $x, y \in X$.
Replacing $y$ by $ry$ in (\ref{pthm1.2}) and using
(\ref{pthm1.1}), we have
\begin{equation}\label{pthm1.3}
r[f(x+sy)+f(x-sy)]=2f(x)-s[f(x+ry)+f(x-ry)]
\end{equation}
for all $x, y \in X$.
Again if we replace $x$ by $sx$ in (\ref{pthm1.3}) and use
(\ref{pthm1.1}), we get
\begin{equation}\label{pthm1.4}
r(1+r)[f(x+y)+f(x-y)]=2(1+r)f(x)-[f(sx+ry)+f(sx-ry)]
\end{equation}
for all $x, y \in X$.
Applying (\ref{Gq}) and using the oddness of $f$, we have
\begin{equation}\label{pthm1.5}
f(sx+ry)+f(sx-ry)=2sf(x)+rs[f(x+y)+f(x-y)]
\end{equation}
for all $x, y$ in $X$.
So it follows from (\ref{pthm1.4}) and (\ref{pthm1.5}) that
\begin{equation}\label{pthm1.6}
f(x+y)+f(x-y)=2f(x)
\end{equation}
for all $x, y$ in $X$.
It easily follows from (\ref{pthm1.6}) that $f$ is additive,
that is, $f(x+y)=f(x)+f(y)$ for all $x, y \in X$.
So if $r$ is a rational number, then $f(rx)=rf(x)$ for all $x$
in $X$.
Therefore it follows from (\ref{pthm1.1}) that $(r^2-r)f(x)=0$
for all $x$ in $X$.
Since $r,\, s$ are nonzero, we infer that $f \equiv 0$.
\end{proof}

\begin{thm}\label{thm2}
Let $X$ and $Y$ be real vector spaces and $f : X \to Y$ be an
even function satisfying (\ref{Gq}).
Then $f$ satisfies (\ref{q}).
\end{thm}
\begin{proof}
Letting $x=y=0$ in (\ref{Gq}), we get $f(0)=0$.
Replacing $x$ by $x+y$ in (\ref{Gq}), we get
\begin{equation}\label{pthm2.1}
f(rx+y)=rf(x+y)+sf(y)-rsf(x)
\end{equation}
for all $x, y \in X$.
Replacing $y$ by $-y$ in (\ref{pthm2.1}) and using the evenness
of $f$, we get
\begin{equation}\label{pthm2.2}
f(rx-y)=rf(x-y)+sf(y)-rsf(x)
\end{equation}
for all $x, y$ in $X$.
Adding (\ref{pthm2.1}) to (\ref{pthm2.2}), we obtain
\begin{equation}\label{pthm2.3}
f(rx+y)+ f(rx-y)=r[f(x+y)+f(x-y)]+2sf(y)-2rsf(x)
\end{equation}
for all $x, y \in X$.
Replacing $y$ by $x+ry$ in (\ref{pthm2.1}), we get
\begin{equation}\label{pthm2.4}
f\big(r(x+y)+x\big)=rf(2x+ry)+sf(x+ry)-rsf(x)
\end{equation}
for all $x, y$ in $X$.
Using (\ref{pthm2.1}) in (\ref{pthm2.4}), by a simple
computation, we get
\begin{equation}\label{pthm2.5}
f(2x+y)+2f(x)+f(y)=2f(x+y)+f(2x)
\end{equation}
for all $x, y$ in $X$.
Putting $y=-x$ in (\ref{pthm2.5}), we get that
$f(2x)=4f(x)$ for all $x \in X$.
Therefore it follows from (\ref{pthm2.5}) that
\begin{equation}\label{pthm2.6}
f(2x+y)+f(y)=2f(x+y)+2f(x)
\end{equation}
for all $x, y$ in $X$.
Replacing $y$ by $y-x$ in (\ref{pthm2.6}), we get that
$f(x+y)+f(y-x)=2f(x)+2f(x)$ for all $x, y \in X$.
So $f$ satisfies (\ref{q}).
\end{proof}

\begin{thm}\label{thm3}
Let  $f : X \to Y$ be a function between real vector spaces
$X$ and $Y$.
If $r$ is a rational number, then $f$ satisfies (\ref{Gq}) if
and only if $f$ satisfies (\ref{q}).
\end{thm}
\begin{proof}
Let $f_o$ and $f_e$ be the odd and the even part of $f$.
Suppose that $f$ satisfies (\ref{Gq}).
It is clear that $f_o$ and $f_e$ satisfy (\ref{Gq}).
By Theorems \ref{thm1} and \ref{thm2}, $f_o \equiv 0$ and
$f_e$ satisfies (\ref{q}).
Since $f = f_o + f_e$, we conclude that $f$ satisfies
(\ref{q}).
\par
Conversely, let $f$ satisfy (\ref{q}).
Then there exists a unique symmetric bi-additive function
$B : X \times X \to Y$ such that $f(x)=B(x,x)$ for all
$x \in X$ (see \cite{Kann95}).
Therefore
\begin{align*}
&rf(x)+sf(y)-rsf(x-y)
\\
&=rB(x,x)+sB(y,y)-rsB(x-y,x-y)
\\
&=r^2B(x,x)+s^2B(y,y)+2rsB(x,y) \quad
(\mbox{$r,\, s$ are rational numbers})
\\
&=B(rx+sy,rx+sy)=f(rx+sy)
\end{align*}
for all $x, y \in X$.
So $f$ satisfies (\ref{Gq}).
\end{proof}

\begin{prop}
Let $\mathcal{X}$ be a linear space with the norm $\| \cdot \|$.
$\mathcal{X}$ is an inner product space if and only if there
exists a real number $0 < r < 1$ such that
\begin{equation}\label{inner0}
\|rx+sy\|^2+rs\|x-y\|^2=r\|x\|^2+s\|y\|^2
\end{equation}
for all $x, y \in \mathcal{X}$, where $s = 1 - r$.
\end{prop}
\begin{proof}
Let $f : \mathcal{X} \to \Bbb{R}$ be a function defined by
$f(x)=\| x \|^2$.
If $\mathcal{X}$ is an inner product space, then $f$ satisfies
(\ref{inner0}) for all $r \in \Bbb{R}$.
Conversely, let $r \in (0,1)$ and the (even) function $f$
satisfy (\ref{inner0}).
So $f$ satisfies (\ref{Gq}).
By Theorem \ref{thm3}, the function $f$ satisfies (\ref{q}),
that is,
$$
\| x+y \|^2 + \| x-y \|^2 = 2\| x \|^2 + 2\| y \|^2
$$
for all $x, y \in \mathcal{X}$.
Therefore $\mathcal{X}$ is an inner product space
(see \cite{JV}).
\end{proof}
\begin{prop}
Let $p,q,u,v\in\Bbb{R}\setminus\{0\}$ and $\mathcal{X}$ be a
linear space with the norm $\| \cdot \|$.
Suppose that
\begin{equation}\label{inner}
    \|rx+sy\|^p+rs\|x-y\|^q=r\|x\|^u+s\|y\|^v
\end{equation}
for all $x,y$ in $\mathcal{X}$, where $0<r<1$ and $s=1-r.$ Then
$p=q=u=v=2.$
\end{prop}
\begin{proof}
Setting $y=0$ in (\ref{inner}), we get
\begin{equation}\label{p.prop.1}
  |r|^p\|x\|^p+rs\|x\|^q=r\|x\|^u
\end{equation}
for all $x$ in $\mathcal{X}$.
If we take $x\in \mathcal{X}$ with $\|x\|=1$ in (\ref{p.prop.1}),
we get that $p=2$.
Letting $y=x$ in (\ref{inner}), we get
\begin{equation}\label{p.prop.2}
  \|x\|^2=r\|x\|^u+s\|x\|^v
\end{equation}
for all $x$ in $\mathcal{X}$. Letting $x=0$ in (\ref{inner}),
we get
\begin{equation}\label{p.prop.3}
  r\|y\|^q=\|y\|^v-s\|y\|^2
\end{equation}
for all $y$ in $\mathcal{X}$. Since $p=2$, it follows from
(\ref{p.prop.1}) and (\ref{p.prop.3}) that
\begin{equation}\label{p.prop.4}
  r\|x\|^u-s\|x\|^v=(r-s)\|x\|^2
\end{equation}
for all $x \in \mathcal{X}$.
Using (\ref{p.prop.2}) and (\ref{p.prop.4}), we get
$\|x\|^u = \|x\|^v$ for all $x \in \mathcal{X}$.
Hence $u=v$ and (\ref{p.prop.2}) implies that $u=v=2$.
Finally, $q=2$ follows from (\ref{p.prop.3}).
\end{proof}
\begin{cor}
Let $\mathcal{X}$ be a linear space with the norm $\| \cdot \|$.
$\mathcal{X}$ is an inner product space if and only if there exists
a real number $0 < r < 1$ and  $p,q,u,v\in\Bbb{R}\setminus\{0\}$
such that
\[
 \|rx+sy\|^p+rs\|x-y\|^q=r\|x\|^u+s\|y\|^v
\]
for all $x, y \in \mathcal{X}$, where $s = 1 - r$.
\end{cor}

\section{Stability of of EQ. (\ref{Gq}) on restricted domains}

In this section, we investigate the Hyers-Ulam stability of
the functional equation (\ref{Gq}) on a restricted domain.
As an application we use the result to the study of an
asymptotic behavior of that equation.
It should be mentioned that Skof \cite{Skof} was the first
author who treats the Hyers-Ulam stability of the quadratic
equation.
Czerwik \cite{Czerwik92} proved a Hyers-Ulam-Rassias
stability theorem on the quadratic equation.
As a particular case he proved the following theorem.
\begin{thm}\label{czerwik}
Let $\delta \ge 0$ be fixed.
If a mapping $f : X \to Y$ satisfies the inequality
$$
\| f(x+y)+f(x-y)-2f(x)-2f(y) \| \le \delta
$$
for all $x, y \in X$, then there exists a unique quadratic
mapping $Q : X \to Y$ such that
$\| f(x)-Q(x) \| \le \frac{\delta}{2}$ for all $x \in X$.
Moreover, if $f$ is measurable or if $f(tx)$ is continuous in
$t$ for each fixed $x \in X$, then $Q(tx)=t^2Q(x)$ for all
$x \in X$ and $t \in \Bbb{R}$.
\end{thm}

We recall that $r,\, s$ are nonzero real numbers with
$r+s=1$.

\begin{thm}\label{thm1.r}
Let $d > 0$ and $\delta \ge 0$ be given.
Assume that an even mapping $f : X \to Y$ satisfies the
inequality
\begin{equation}\label{thm1.r.1}
\| f(rx+sy)+rsf(x-y)-rf(x)-sf(y) \| \le \delta
\end{equation}
for all $x, y \in X$ with $\| x \| + \| y \| \ge d$.
Then there exists $K > 0$ such that $f$ satisfies
\begin{equation}\label{thm1.r.2}
\| f(x+y)+f(x-y)-2f(x)-2f(y) \|
\le \frac{4(2+|r|+|s|)}{|rs|} \delta
\end{equation}
for all $x, y \in X$ with $\| x \| + \| y \| \ge K$.
\end{thm}
\begin{proof}
Let $x, y \in X$ with $\|x\|+\|y\| \ge 2d$.
Then, since $\| x+y \| + \| y \| \ge \max\{ \|x\|, 2\|y\|-\|x\| \}$,
we get $\| x+y \| + \| y \| \ge d$.
So it follows from (\ref{thm1.r.1}) that
\begin{equation}\label{pthm1.r.1}
\| f(rx+y)+rsf(x)-rf(x+y)-sf(y) \| \le \delta
\end{equation}
for all $x, y \in X$ with $\| x \| + \| y \| \ge 2d$.
So
\begin{equation}\label{pthm1.r.2}
\| f(ry+x)+rsf(y)-rf(x+y)-sf(x) \| \le \delta
\end{equation}
for all $x, y \in X$ with $\| x \| + \| y \| \ge 2d$.

Let $x, y \in X$ with $\| x \| + \| y \| \ge 4d \Big(
\frac{1}{|r|} + \big| 1 - \frac{1}{|r|} \big| \Big)$.
We have two cases: \\
\textbf{Case I.}
$\| y \| > \frac{2d}{|r|}$.
Then $\| x \| + \| x+ry \| \ge |r| \| y \| \ge 2d$. \\
\textbf{Case II.}
$\| y \| \le \frac{2d}{|r|}$.
Then we have $\| x \| \ge 2d \Big( \frac{1}{|r|} +
2 \big| 1 - \frac{1}{|r|} \big| \Big)$.
So
$$
\| x \| + \| x+ry \| \ge 2\| x \| - |r| \| y \|
\ge 2d \Big( \frac{2}{|r|} + 4 \big| 1 - \frac{1}{|r|} \big|
- 1 \Big) \ge 2d.
$$
Therefore we get that $\| x \| + \| x+ry \| \ge 2d$ from cases
I and II.
Hence by (\ref{pthm1.r.1}) we have
\begin{equation}\label{pthm1.r.3}
\| f\big( r(x+y)+x \big) + rsf(x) - rf(2x+ry) - sf(x+ry) \|
\le \delta
\end{equation}
for all $x, y \in X$ with $\| x \| + \| y \| \ge 4d \Big(
\frac{1}{|r|} + \big| 1 - \frac{1}{|r|} \big| \Big)$.
Set $M := 4d \Big( \frac{1}{|r|} + \big| 1 - \frac{1}{|r|}
\big| \Big)$.
Then
$$
\| x+y \| + \| x \| \ge \frac{M}{2} \ge 2d, \quad
\| 2x \| + \| y \| \ge M \ge 4d
$$
for all $x, y \in X$ with $\| x \| + \| y \| \ge M$.
From (\ref{pthm1.r.1}) and (\ref{pthm1.r.2}), we get the
following inequalities:
\begin{align*}
& \| f\big( r(x+y)+x \big) + rsf(x+y) - rf(2x+y) - sf(x) \|
  \le \delta, \\
& \| rf(ry+2x) + r^2 sf(y) - r^2 f(2x+y) - rsf(2x) \|
  \le \delta |r|, \\
& \| sf(ry+x) + rs^2 f(y) - rsf(x+y) - s^2 f(x) \|
  \le \delta |s|.
\end{align*}
Using (\ref{pthm1.r.3}) and the above inequalities, we get
\begin{equation}\label{pthm1.r.4}
\| f(2x+y)+2f(x)+f(y)-2f(x+y)-f(2x) \|
\le \frac{2+|r|+|s|}{|rs|} \delta
\end{equation}
for all $x, y \in X$ with $\| x \| + \| y \| \ge M$.
If $x, y \in X$ with $\| x \| + \| y \| \ge 2M$, then
$\| x \| + \| y-x \| \ge M$.
So it follows from (\ref{pthm1.r.4}) that
\begin{equation}\label{pthm1.r.5}
\| f(x+y)+2f(x)+f(y-x)-2f(y)-f(2x) \|
\le \frac{2+|r|+|s|}{|rs|} \delta.
\end{equation}
Letting $y=0$ in (\ref{pthm1.r.5}), we get
\begin{equation}\label{pthm1.r.6}
\| 4f(x)-f(2x)-2f(0) \| \le \frac{2+|r|+|s|}{|rs|} \delta
\end{equation}
for all $x, y \in X$ with $\| x \| \ge 2M$.
Letting $x=0$ (and $y \in X$ with $\| y \| \ge 2M$) in
(\ref{pthm1.r.5}), we get $\| f(0) \| \le \frac{2+|r|+|s|}{|rs|}
\delta$.
Therefore it follows from (\ref{pthm1.r.5}) and
(\ref{pthm1.r.6}) that
\begin{equation}\label{pthm1.r.7}
\begin{aligned}
\| & f(x+y)+f(y-x)-2f(x)-2f(y) \| \\
   & \le \| f(x+y)+2f(x)+f(y-x)-2f(y)-f(2x) \| \\
   & \quad + \| 4f(x)-f(2x)-2f(0) \| + 2 \| f(0) \| \\
   & \le \frac{4(2+|r|+|s|)}{|rs|} \delta
\end{aligned}
\end{equation}
for all $x, y \in X$ with $\| x \| \ge 2M$.
Since $f$ is even, the inequality (\ref{pthm1.r.7}) holds for
all $x, y \in X$ with $\| y \| \ge 2M$.
Therefore
$$
\| f(x+y)+f(x-y)-2f(x)-2f(y) \|
\le \frac{4(2+|r|+|s|)}{|rs|} \delta
$$
for all $x, y \in X$ with $\| x \| + \| y \| \ge 4M$.
This completes the proof by letting $K := 4M$.
\end{proof}

\begin{thm}\label{thm2.r}
Let $d > 0$ and $\delta \ge 0$ be given.
Assume that an even mapping $f : X \to Y$ satisfies the
inequality (\ref{thm1.r.1}) for all $x, y \in X$ with
$\| x \| + \| y \| \ge d$.
Then $f$ satisfies
\begin{equation}\label{thm2.r.1}
\| f(x+y)+f(x-y)-2f(x)-2f(y) \|
\le \frac{19(2+|r|+|s|)}{|rs|} \delta
\end{equation}
for all $x, y \in X$.
\end{thm}
\begin{proof}
By Theorem \ref{thm1.r} there exists $K > 0$ such that $f$
satisfies (\ref{thm1.r.2}) for all $x, y \in X$ with
$\| x \| + \| y \| \ge K$ and
$\| f(0) \| \le \frac{2+|r|+|s|}{|rs|} \delta$ (see the proof
of Theorem \ref{thm1.r}).
Using Theorem 2 of \cite{{JRassias02}}, we get that
\begin{align*}
\| f(x+y)+f(x-y)-2f(x)-2f(y) \|
& \le \frac{18(2+|r|+|s|)}{|rs|} \delta + \| f(0) \| \\
& \le \frac{19(2+|r|+|s|)}{|rs|} \delta
\end{align*}
all $x, y \in X$.
\end{proof}

\begin{thm}\label{thm3.r}
Let $d > 0$ and $\delta \ge 0$ be given.
Assume that an even mapping $f : X \to Y$ satisfies the
inequality (\ref{thm1.r.1}) for all $x, y \in X$ with
$\| x \| + \| y \| \ge d$.
Then there exists a unique quadratic mapping $Q : X \to Y$
such that $Q(x)=\lim_{n\to\infty}4^{-n}f(2^n x)$ and
\begin{equation}\label{thm3.r.1}
\| f(x)-Q(x) \| \le \frac{19(2+|r|+|s|)}{2|rs|} \delta
\end{equation}
for all $x \in X$.
\end{thm}
\begin{proof}
The result follows from Theorems \ref{czerwik} and \ref{thm2.r}.
\end{proof}

Skof \cite{Skof} has proved an asymptotic property of the
additive mappings and Jung \cite{Jung98} has proved an
asymptotic property of the quadratic mappings (see also
\cite{jungk}).
We prove  such a property also for the quadratic mappings.

\begin{cor}
An even mapping $f : X \to Y$ satisfies (\ref{Gq}) if and
only if the asymptotic condition
\begin{equation}\label{cor.asy}
\| f(rx+sy)+rsf(x-y)-rf(x)-sf(y) \| \to 0 \quad
\mbox{as}~ \| x \| + \| y \| \to \infty
\end{equation}
holds true.
\end{cor}
\begin{proof}
By the asymptotic condition (\ref{cor.asy}), there exists a
sequence $\{ \delta_n \}$ monotonically decreasing to $0$ such
that
\begin{equation}\label{p.cor.asy.1}
\| f(rx+sy)+rsf(x-y)-rf(x)-sf(y) \| \le \delta_n
\end{equation}
for all $x, y \in X$ with $\| x \| + \| y \| \ge n$.
Hence, it follows from (\ref{p.cor.asy.1}) and Theorem
\ref{thm3.r} that there exists a unique quadratic mapping
$Q_n : X \to Y$ such that
\begin{equation}\label{p.cor.asy.2}
\| f(x)-Q_n(x) \| \le \frac{19(2+|r|+|s|)}{2|rs|} \delta_n
\end{equation}
for all $x \in X$.
Since $\{ \delta_n \}$ is a monotonically decreasing sequence,
the quadratic mapping $Q_m$ satisfies (\ref{p.cor.asy.2}) for
all $m \ge n$.
The uniqueness of $Q_n$ implies $Q_m=Q_n$ for all $m \ge n$.
Hence, by letting $n \to \infty$ in (\ref{p.cor.asy.2}), we
conclude that $f$ is quadratic.
\end{proof}

\begin{cor}
Let $r$ be rational.
An even mapping $f : X \to Y$ is quadratic if and only if the
asymptotic condition (\ref{cor.asy}) holds true.
\end{cor}

{\small

}
\end{document}